\documentclass[11pt]{article}
\usepackage[utf8]{inputenc}
\usepackage[a4paper]{geometry}
\usepackage{natbib}
\usepackage{graphicx}
\usepackage{amsmath,amssymb,amsthm}
\usepackage[colorinlistoftodos,textsize=tiny]{todonotes} 
\usepackage{color}
\usepackage{hyperref}

\newcommand{\Prob}[1]{\mathbb{P}\left(#1\right)}
\newcommand{\E}[1]{\mathbb{E}\left[#1\right]}
\newcommand{\Var}[1]{\mathrm{Var}\left(#1\right)}

\newcommand{\mcal}[1]{\mathcal{#1}}

\newcommand{\Ind}[1]{I\!\left\{ #1 \right\}}

\newtheorem{corollary}{Corollary}[section]
\newtheorem{definition}{Definition}[section]
\newtheorem{theorem}{Theorem}[section]
\newtheorem{proposition}{Proposition}[section]
\newtheorem{lemma}{Lemma}[section]
\theoremstyle{remark}
\newtheorem{example}{Example}[section]
\theoremstyle{remark}
\newtheorem{remark}{Remark}[section]

\title{The semi-Markov beta-Stacy process: a Bayesian non-parametric prior for semi-Markov processes.}
\author{Andrea Arf\`e\footnote{Harvard-MIT Center for Regulatory Science, Harvard Medical School, 02155 Boston, Massachusetts.} , Stefano Peluso\footnote{Department of Statistics and Quantitative Methods, University of Milano-Bicocca, 20126 Milan, Italy.} , 
Pietro Muliere\footnote{Department of Decision Sciences, Bocconi University, 20136 Milan, Italy.}}

\begin{document}
\maketitle

\begin{abstract}
The literature on Bayesian methods for the analysis of discrete-time semi-Markov processes is sparse. In this paper, we introduce the \emph{semi-Markov beta-Stacy process}, a stochastic process useful for the Bayesian non-parametric analysis of semi-Markov processes. The semi-Markov beta-Stacy process is conjugate with respect to data generated by a semi-Markov process, a property which makes it easy to obtain probabilistic forecasts. Its predictive distributions are characterized by a reinforced random walk on a system of urns.

\textbf{Running head:} Semi-Markov beta-Stacy. 

\end{abstract}

\section{Introduction}

Discrete-time Semi-Markov processes generalize Markov chains by allowing the \emph{holding times}, the times spent in each visited state, to have distributions other than the geometric \citep{Cinlar1969}. We address how to perform inferences and predictions for these processes from a Bayesian non-parametric perspective.

Semi-Markov processes are used to predict many phenomena. Applications include time-series and longitudinal data analysis \citep{Bulla2006}, survival analysis and reliability \citep{Barbu2004, Mitchell2011}, finance and actuarial sciences \citep{Janssen2007}, and biology \citep{Barbu2009}. 


Despite their usefulness (and in contrast with the continuous-time case; c.f. \citealp{Phelan1990, Bulla2007, Zhao2013}), the literature on inferential approaches for discrete-time semi-Markov processes is sparse \citep[Chapter 4]{Barbu2009}. The current literature focuses on processes with a finite state space. From the frequentist perspective, \citet{Satten1999} and \citet{Barbu2009} study non-parametric estimators of the transition probabilities and the holding times distributions. From the Bayesian perspective, specific parametric models have been used in different settings \citep{Patwardhan1980,Schiffman2007,Masala2013,Mitchell2011}, but no general non-parametric approach has been developed.  

We introduce the \emph{semi-Markov beta-Stacy process}, a non-parametric prior for the Bayesian analysis of semi-Markov models. The semi-Markov beta-Stacy can be used both for processes with a finite or countably infinite state space. We will show that this prior is conjugate with respect to the observation of one or more processes for a fixed time - a property that facilitates inferences and predictions for semi-Markov processes. 

The one-step-ahead predictive laws of the semi-Markov beta-Stacy process are equal to the transition kernels of a \emph{reinforced semi-Markov process}. This new process is the discrete-time analogue of that of \citet{Muliere2003} and \citet{Bulla2007}. Here, ``reinforcement'' means that each time a state is visited, this becomes more likely to be visited again \citep{Coppersmith1986, Pemantle1988,Pemantle2007}. 

We characterize the semi-Markov beta-Stacy process using a \emph{reinforced urn process}, i.e. a random walk over a system of reinforced urns. In such processes, whenever the  walk visits an urn, a ball is extracted and replaced by additional balls of the same color. The random walk then jumps to another urn determined by the extracted color. These processes are increasingly being used to construct nonparametric priors for a wide range of stochastic models \citep{Blackwell1973, Doksum1974, Mauldin1992, Walker1997, Muliere2000,Muliere2003, Bulla2007,Fortini2012,Peluso2015, Caron2017, Arfe2018a}.


Before proceeding, we introduce some notation. First, if $F$ is a non-decreasing function on the integers (adjoined with the $\sigma$-algebra of all subsets), then the symbol $F$ will also represent the associated measure. Hence, for example, $F(b)-F(a)=F((a,b])$ for all $a<b$, where $(a,b]$ is the set of all integers $x$ such that $a<x\leq b$. In particular, if $F$ is a probability distribution, we will use the same symbol to denote the corresponding cumulative distribution function, i.e. we set $F(x)=F((-\infty, x])$ for all integers $x$. Second, if $x=(x_1,x_2,\ldots)$ is a finite or infinite sequence, we denote with $x_{a:b}$ either the subsequence $(x_a,\ldots,x_b)$ of length $b-a+1$ if $a\leq b$, or the empty sequence of length $0$ if $a>b$. Lastly, empty sums and products are respectively defined to equal 0 and 1. 

The paper is structured as follows. In Section \ref{sec:semimarkov} we define discrete-time semi-Markov processes. In Section \ref{sec:smbs} we introduce the semi-Markov beta-Stacy process prior. In Section \ref{sec:post} we derive the corresponding posterior distributions and show that this process prior is conjugate. In Section \ref{sec:pred} we introduce reinforced semi-Markov process and show that these describe the predictive distributions of the semi-Markov beta-Stacy process prior. In Section \ref{sec:urns} we characterize the semi-Markov beta-Stacy process using a system of reinforced urns. In Section \ref{sec:applications} we illustrate the semi-Markov beta-Stacy process prior in a simulation study. Lastly, in Section \ref{sec:discussion} we provide some concluding remarks. 

\section{Semi-Markov processes: definition and basic properties} \label{sec:semimarkov}

Let $E$ be a non-empty finite or countably infinite set. Let $\mathbf{P}=(P^{i,j})_{i,j\in E}$ be a transition matrix on $E$ such that $P^{i,i}=0$ for all $i\in E$ and let $\mathbf{F}=(F^i(\cdot):i\in E)$ be a collection of probability distributions on the set of positive integers. Fix a $l_0$ in $E$.

\begin{definition}\label{def:renewal}
A \emph{Markov renewal process} starting at $l_0$ is a stochastic process $(L,T)=(L_n,T_n)_{n\geq 0}$ such that $\Prob{L_0=l_0|(\mathbf{P},\mathbf{F})} = 1$ and 
\begin{equation*}
\Prob{L_{n+1}=j,T_{n}\leq t | L_n=i,L_{0:n-1},T_{0:n-1},(\mathbf{P},\mathbf{F})} = F^{i}(t)P^{i,j}
\end{equation*}
for all integers $n\geq 0$, $t\geq 1$, and all $i,j\in E$. The pair $(\mathbf{P},\mathbf{F})$ is the \emph{characteristic couple} of $(L,T)$. Suppressing the dependence on $l_0$, we write $(L,T)\sim MR(\mathbf{P},\mathbf{F})$.  
\end{definition}

If $(L,T)\sim MR(\mathbf{P},\mathbf{F})$, define $\tau_0=0$ and $\tau_{n+1}=\sum_{h=0}^n T_h$ for all $n\geq 0$. Also let $N(t)=\sum_{n=1}^{+\infty}\Ind{\tau_n \leq t}$ for all integers $t\geq 0$.

\begin{definition}\label{def:semimark}
The process $(S_t)_{t\geq 0}$ defined by $S_t=L_{N(t)}$ is the \emph{semi-Markov Process} associated with $(L,T)$, $S=(S_t)_{t\geq 0}\sim SM(\mathbf{P},\mathbf{F})$ in symbols. The times $(\tau_n)_{n\geq 1}$ are the \emph{jump times} of $S$.
\end{definition}

A semi-Markov process $(S_t)_{t\geq 0}$ describes the evolution of some system as it goes through different states. The elements of $E$ represent the possible states. Additionally, $S_t$ is the state occupied at time $t$, $N(t)$ is the number of state changes occurred up to time $t$, $\tau_n$ is the time of the $n$-th state change, and $T_k$ is the length of time the system spends in its $k$-th state. These interpretations are possible because  $P^{i,i}=0$ for all $i\in E$ implies that $L_k\neq L_{k+1}$ for all $k$ with probability 1.

\begin{example}\label{ex:factory}
\citet[Sections 3.4]{Barbu2009} use a semi-Markov model to describe the operation of a factory. The factory waste is treated in a disposal unit. If the disposal unit fails, waste is stored in a tank. If the disposal unit is repaired before the tank is full, the factory continues to operate and the tank is purged. Otherwise, the factory stops and some time is necessary to restart it. The state space is thus $E=\{1,2,3\}$: $1$ represents that factory is operational, $2$ represents that the disposal unit is malfunctioning but the factory is still operational, and $3$ represents that the factory is stopped. Hence, $P^{1,3}=P^{3,2}=0$. Moreover, $F^{1}(\cdot)$ is the distribution of the time until the next disposal unit failure, $F^{2}(\cdot)$ that of the time until a malfunctioning disposal unit is either restored or when it fully breaks down, and $F^{3}(\cdot)$ that needed to restart the factory.
\end{example}

\begin{remark}\label{rem:equivSM}
Denote $l(t)=\max\{k=0,1,\ldots,t : S_t=S_{t-1}=\cdots=S_{t-k}\}$ 
the time spent by $S$ in the state $S_t$ just prior to time $t$. If $N(t)=n$, to observe $S_{0:t}$ is the same as to know $L_{0:n}$, $T_{0:n-1}$, and that $T_{n}> l(t)$. The sequence $L_{0:n}$ is the collection of the distinct states in $S_{0:t}$ in order of appearance. Instead, $T_{0:n-1}$ is determined by the position of the state changes in the sequence $S_{0:t}$. The value of $T_n$ is \emph{censored}, as it is only known to exceed a known threshold \citep{Kalbfleisch2002}.
\end{remark}

\begin{example}\label{ex:equivSM}
Observing $S_{0:5}=(i_0,i_0,i_1,i_2,i_2,i_2)$ for some distinct $i_0,i_1,i_2\in E$ is equivalent to observing $N(5)=2$, $L_0=i_0$, $L_1=i_1$, $L_2=i_2$, $\tau_1=2$, $\tau_2=3$, $\tau_3>5$, $l(5)=2$, $T_0=2$, $T_1=1$, and $T_2\geq 3$, i.e. $T_2>2=l(5)$.
\end{example}

Note that, since $P^{i,i}=0$ and $F^i$ has support on the positive integers for all $i\in E$, the semi-Markov process $S$ cannot have absorbing states, i.e. states such that $S_t=i$ for all sufficiently large $t\geq 0$ with positive probability. This assumption simplifies our analysis, although it might be restrictive for some applications. An absorbing state $i$ could be allowed by letting $P^{i,i}=1$ and $F^i(\{+\infty\})=1$. With additional effort, the results in the following sections could be extended to this case as well. 

\section{The semi-Markov beta-Stacy process prior} \label{sec:smbs}

The realizations of a semi-Markov beta-Stacy process are random characteristic couples $(\mathbf{P},\mathbf{F})$. Thus, the semi-Markov beta-Stacy process is a nonparametric prior distribution on $(\mathbf{P},\mathbf{F})$ \citep{Ferguson1973}. To define it, we will separately assign a prior distribution to i) each holding time distribution $F^i$ and ii) the transition matrix $\mathbf{P}$. 

We first consider the discrete-time beta-Stacy process of \citet{Walker1997}, a common Bayesian nonparametric prior for time-to-event distributions \citep{Singpurwalla2006, Bulla2007, Arfe2018a}. The beta-Stacy process will be used as the prior for the holding time distributions $F^i$.

\begin{definition}[\citet{Walker1997}]\label{def:betastacy}
Let $c(t)>0$ all integers $t>0$. Also let $F_0$ be a probability distributions with support on the positive integers. A random distribution $F$ is a \emph{beta-Stacy process} $BS(c,F_0)$ if there is a sequence $(U_t)_{t\geq 1}$ of independent random variables such that i) $U_t\sim Beta(c(t)F_0(\{t\}),c(t)F_0((t,+\infty)))$ for all integers $t\geq 1$ and ii) $F((t,+\infty))=\prod_{k=1}^{t}(1-U_k)$ for all integers $t\geq 0$.
\end{definition}

\begin{remark}
If $F\sim BS(c,F_0)$, then $\E{F(t)}=F_0(t)$ and $\Var{F(t)}$ is a decreasing function of $c(t)$ such that $\Var{F(t)}\rightarrow 0$ as $c(t)\rightarrow +\infty$. Hence $F_0$ is the mean of the process, while $c$ controls its dispersion \citep{Walker1997}. 
\end{remark}

The beta-Stacy process is conjugate, i.e. the distribution of $F\sim BS(c,F_0)$ conditional on a sample of exact observations from $F$ is again a beta-Stacy process. The beta-Stacy process is also conjugate with respect to censored observations (c.f. Remark \ref{rem:equivSM}). These properties are summarized in the following Proposition, which is a special case of Theorem 1 of \citet{Walker1997}.  
\begin{proposition}[\citet{Walker1997}]\label{prop:postbetastacy}
If $F\sim BS(c,F_0)$ and $T_1,\ldots,T_n$ are independently distributed according to $F$, then the posterior distribution of $F$ given $T_1,\ldots,T_n$ is $BS(c_*,F_*)$, where
\begin{align*}
F_*((t,+\infty)) &= \prod_{s=1}^t \left[ 1-\frac{c(s)F_0(\{s\})+N(\{s\})}{c(s)F_0([s,+\infty))+N([s,+\infty))} \right] \\
c_*(t) &= \frac{c(t)F_0((t,+\infty))+N((t,+\infty))}{F_*((t,+\infty))}.
\end{align*}
where $N(t)=\sum_{i=1}^n\Ind{T_i\leq t}$. Instead, the posterior distributions of $F$ given $T_n>t^*$ (i.e. a censored observation), where $t^*$ is a fixed constant, is $BS(c_*,F_*)$, where now
\begin{align*}
F_*((t,+\infty)) &= \prod_{s=1}^t \left[ 1-\frac{c(s)F_0(\{s\})}{c(s)F_0([s,+\infty))+\Ind{t^*\geq s}} \right], \\
c_*(t) &= \frac{c(t)F_0((t,+\infty))+\Ind{t^*\geq s}}{F_*((t,+\infty))}.
\end{align*} 
\end{proposition}

To specify a prior on the transition matrix $\textbf{P}$ we will use the Dirichlet process, a common non-parametric process prior for probability measures \citet{Ferguson1973, Hjort2010}. Since the $i$-the row ${P}^i=(P^{i,j})_{j\in E}$ of $\mathbf{P}$ is the measure ${P}^i(\cdot)$ on $E$ defined by ${P}^i(\{j\})=P^{i,j}$ for all $j\in E$, this can be assigned a Dirichlet process prior. 

\begin{definition}[\citet{Ferguson1973}]
Let $m$ be a measure on $E$ such that $0<m(E)<+\infty$. A random probability measure $P$ on $E$ is a \emph{Dirichlet process} with \emph{base measure} $m$, or $P\sim Dir(m)$ in symbols, if for every partition $A_1,\ldots,A_n$ of $E$ it holds that 
$$(P(A_1),\ldots,P(A_n))\sim Dirichlet(m(A_1),\ldots,m(A_n)).$$
\end{definition}
\begin{remark}\label{rem:dirichlet}
If $P\sim Dir(m)$, then $P(A)\sim Beta(m(A),m(A^c))$ for all $A\subseteq E$. Hence, $\E{P(A)}=m(A)/m(E)$ and $\Var{P(A)}
\rightarrow 0$ as $m(E)\rightarrow +\infty$. Additionally, if $m(A)=0$, then $P(A)=0$ almost surely.
\end{remark}


The Dirichlet process is also conjugate.

\begin{proposition}[Theorem 1, \citet{Ferguson1973}]\label{prop:postdirichlet}
Suppose $P\sim Dir(m)$ and that $X_1,$ $\ldots,$ $X_n$ are independent with common law $P$. The distribution of $P$ given $X_1,\ldots,X_n$ is $Dir(m_*)$, where $m_*$ is defined by $m_*(\{i\})=m(\{i\})+\sum_{j=1}^n \Ind{X_j=i}$ for all $i\in E$.
\end{proposition}

We can now define the semi-Markov beta-Stacy process. Let $m^i(\cdot)$ be a measure on $E$ such that $0<m^i(E)<+\infty$ and $m^i(\{i\})=0$ for all $i\in E$. Let $c^i(t)>0$ for all integers $t>0$. Also let $F_0^i$ be a distribution with support on the positive integers for all $i\in E$. Lastly, let $\mathbf{m}=(m^i)_{i\in E}$, $\mathbf{c}=(c^i)_{i\in E}$, and $\mathbf{F_0}=(F_0^i)_{i\in E}$.

\begin{definition}\label{def:multbetastacy}
A random characteristic couple $(\mathbf{P},\mathbf{F})$ has a \emph{semi-Markov beta-Stacy} distribution with parameters $(\mathbf{m},\mathbf{c},\mathbf{F_0})$, or $(\mathbf{P},\mathbf{F}) \sim SMBS(\mathbf{m},\mathbf{c},\mathbf{F_0})$, if:
\begin{enumerate}
\item $\mathbf{P}$ and $\mathbf{F}$ are independent;
\item the rows $P^i(\cdot)$, $i\in E$, of $\mathbf{P}$ are independent;
\item the distributions $F^i$, $i\in E$, in $\mathbf{F}$ are independent;
\item $P^i(\cdot)$ is a Dirichlet process with base measure $m^i$ for all $i\in E$: $P^i(\cdot)\sim Dir(m^i)$;
\item for all $i\in E$, $F^i$ is a beta-Stacy process with precision parameters $c^i$ and centering distribution $F_0^i$: $F^i\sim BS(c^i,F_0^i)$.
\end{enumerate}
\end{definition}

Note that each realization of $(\mathbf{P},\mathbf{F}) \sim SMBS(\mathbf{m},\mathbf{c},\mathbf{F_0})$ is a valid characteristic couple, justifying the use of the law of a semi-Markov beta-Stacy process as a prior distribution for a characteristic couple $(\mathbf{P},\mathbf{F})$. 
In general, it will be $P^i(\{j\})=P^{i,j}=0$ almost surely for all $j\in E$ such that $m^i(\{j\})=0$. In this case, each realization of a $SMBS(\mathbf{m},\mathbf{c},\mathbf{F_0})$ will be the characteristic couple of a semi-Markov process which cannot perform transition from $i$ to $j$. 

\section{Posterior computations} \label{sec:post}

We will compute the posterior distribution associated to the beta-Stacy process. This will be another beta-Stacy process, showing that it is conjugate.  

We introduce two additional notations.  Suppose $S_{0:t}$ is observed for some fixed $t>0$. Then, for all $i,j\in E$, $i\neq j$, let $M^{i,j}(t) = \sum_{k=1}^t \Ind{S_{k-1}=i, S_k=j}$ be the number of transitions from state $i$ to state $j$ in $S_{0:t}$. Moreover, let $N^{i,t}(l)$ be the number of visits to the state $i$ of length less or equal than $l$ that are observed in $S_{0:t}$. By Remark \ref{rem:equivSM}, if $N(t)=n$, then $N^{i,t}(l)=\sum_{k=0}^{n-1}\Ind{T_k\leq l, L_k=i}$, the number of holding times for state $i$ observed up to time $t$ and that: i) are not censored, and ii) do not exceed $l$.

\begin{example}
In Example \ref{ex:equivSM}, it is $N^{i_0,5}(1)=0$, $N^{i_0,5}(2)=N^{i_0,5}(l)=1$ for all $l\geq 2$; $N^{i_1,5}(1)=N^{i_1,5}(l)=1$ for all $l\geq 1$; $N^{i,5}(l)=0$ for all $l>0$ if $i\neq i_0,i_1$; and $M^{i_0,i_1}(5)=M^{i_1,i_2}(5)=1$. 
\end{example}

With these notations, we can now state the following theorem. Its proof uses the fact that the beta-Stacy process is conjugate with respect to censored observations.
\begin{theorem}\label{thm:conjugacy}
Suppose that $S\sim SM(\mathbf{P},\mathbf{F})$ and $(\mathbf{P},\mathbf{F})\sim SMBS(\mathbf{m},\mathbf{c},\mathbf{F_0})$. Then, the posterior distribution of $(\mathbf{P},\mathbf{F})$ given $S_{0:t}=i_{0:t}$ is $SMBS(\mathbf{m_*},\mathbf{c_*},\mathbf{F_*})$, where:
\begin{enumerate}
\item For all $i\in E$, $m_*^{i}$ is defined by $m^{i}_*(\{j\})=m^i(\{j\})+M^{i,j}(t)$ for $j\in E$, $j\neq i$.
\item For all $i\in E$, $i\neq i_t$, $F^{i}_*$ and $c^{i}_*$ are determined by letting
\begin{align*}
F^{i}_*((u,+\infty)) &= \prod_{s=1}^u \left[ 1-\frac{c^i(s)F_0^i(\{s\})+N^{i,t}(\{s\})}{c^i(s)F_0^i([s,+\infty))+N^{i,t}([s,+\infty))} \right] \\
c^{i}_*(u) &= \frac{c^i(u)F_0^i((u,+\infty))+N^{i,t}((u,+\infty))}{F_*((u,+\infty))}
\end{align*}
for each integer $u>0$.
\item For $i=i_t$, $F^{i}_*$ and $c^{i}_*$ are instead determined by letting
\begin{align*}
F_*^{i}((u,+\infty)) &= \prod_{s=1}^u \left[ 1-\frac{c^i(s)F_0^i(\{s\})+N^{i,t}(\{s\})}{c^i(s)F_0^i([s,+\infty))+N^{i,t}([s,+\infty))+\Ind{l(t)\geq s}} \right] \\
c^{i}_*(u) &= \frac{c^i(u)F_0^i((u,+\infty))+N^{i,t}((u,+\infty))+\Ind{l(t)\geq u}}{F_*((u,+\infty))}
\end{align*}
for each integer $u>0$.
\end{enumerate}
\end{theorem}
\begin{proof}
By Remark \ref{rem:equivSM}, the likelihood function associated with $S_{0:t}=i_{0:t}$ is 
\begin{align*}
\Prob{S_{0:t}=i_{0:t}|\mathbf{P},\mathbf{F}} &= \Prob{L_{0:n}=l_{0:n}, T_{0:n-1}=t_{0:n-1}, T_n>l(t)|\mathbf{P},\mathbf{F}} \\
&\mkern-18mu\mkern-18mu\mkern-18mu =F^{l_0}(t_0)\left[\prod_{k=1}^{n-1}F^{l_k}(\{t_k\})P^{l_{k-1},l_{k}}\right]\cdot\left[F^{l_{n}}((l(t),+\infty))\right] \\
&\mkern-18mu\mkern-18mu\mkern-18mu = \left[\prod_{i\in E} \prod_{s=1}^t F^i(\{s\})^{N^{i,t}(\{s\})}F^{i}((l(t),+\infty))^{\Ind{i=l_n}}\right] \cdot  \left[\prod_{\substack{i,j\in E\\ i\neq j}} P^{i}(\{j\})^{M^{i,j}(t)}\right]
\end{align*}
The likelihood is the product of individual terms depending only on $F^i$ or $P^i(\cdot)$ for some $i$. Hence, by points 1-3 of Definition \ref{def:multbetastacy}, conditional on $S_{0:t}=i_{0:t}$, $\mathbf{P}$ and $\mathbf{F}$ are independent, the rows $P^i(\cdot)$, $i\in E$, of $\mathbf{P}$ are independent, and the distributions $F^i$, $i\in E$, in $\mathbf{F}$ are independent. Moreover, i) the posterior distribution of $F^i$, $i\neq l_n$, depends only on those $T_k$ such that $L_k=i$, and it is the same as if these were an independent sample from $F^i$; ii) the same is true for the posterior distribution of $F^{l_n}$, except that $T_n$ is censored (only $T_n>l(t)$ is known); iii) the posterior distribution of $P^i(\cdot)$ depends only on those states in $l_{0:n}$ that are preceded by the state $i$, and it is the same as if these states were an independent sample from $P^i(\cdot)$. The thesis now follows from Propositions \ref{prop:postbetastacy} and \ref{prop:postdirichlet}.
\end{proof}

This theorem characterizes the posterior law of $(\mathbf{P},\mathbf{F})$ conditional on the history of a single process $S\sim SM(\mathbf{P},\mathbf{F})$, observed for a fixed amount of time $t$. It shows that the posterior law of $F^i$ depends only on the lengths of the visits to state $i$ observed in the process history $S_{0:t}$ (if $i=S_t$, the duration of the last visit to $i$ in $S_{0:t}$ is censored and only known to exceed $l(t)$). Instead, the posterior of law of $P^{i,j}$ depend only on the number of transitions from $i$ to $j$ observed in $S_{0:t}$.

Theorem \ref{thm:conjugacy} can also be used to compute the posterior law of $(\mathbf{P},\mathbf{F})$ in more general settings. In particular, the Theorem also holds when $t$ could random instead of fixed, provided this is either a stopping time of $S$ or independent of $S$ and $(\mathbf{P},\mathbf{F})$ \citep{Heitjan1991}. In addition, instead of only one, multiple independent processes $S^1,\ldots,S^n\sim SM(\mathbf{P},\mathbf{F})$ may be observed up to times $t^1,\ldots,t^n$ (e.g. because the health status of $n$ patients is observed in parallel; c.f. \citealp{Mitchell2011}). Here, the posterior law of $(\mathbf{P},\mathbf{F})$ given $S^1_{1:t^1},\ldots,S^n_{1:t^n}$ is obtained by applying Theorem \ref{thm:conjugacy} iteratively. 

\section{Predictive laws and reinforced semi-Markov processes} \label{sec:pred}

Assuming $(\mathbf{P},\mathbf{F})\sim SMBS(\mathbf{m},\mathbf{c},\mathbf{F_0})$, we will now derive the one-step-ahead predictive distributions of $S$, i.e. the conditional distributions $\Prob{S_{t+1}=\cdot|S_{0:t}}$ for $t\geq 0$. These play an important role in applications. For instance, in Example \ref{ex:factory}, they quantify the future risk that the factory will have to stop. 

Suppose that $S\sim SM(\mathbf{P},\mathbf{F})$ and $(\mathbf{P},\mathbf{F})\sim SMBS(\mathbf{m},\mathbf{c},\mathbf{F_0})$. Define for simplicity $x(t)=l(t)+1$ for all integers $t\geq 0$. 

\begin{theorem}\label{thm:pred}
For all $t\geq 0$, $\Prob{S_{t+1}=\cdot|S_{0:t}}=k_t(S_{0:t};\cdot)$, where, if $S_t=i\neq j$, it is
\begin{align*}
k_t(S_{0:t};i)&= \frac{c^{i}(x(t))F_0^i((x(t),+\infty))+N^{i,t}((x(t),+\infty))}{c^i(x(t))F_0^i([x(t),+\infty))+N^{i,t}([x(t),+\infty))}, \\
k_t(S_{0:t};j)&= \frac{c^i(x(t))F_0^i(\{x(t)\})+N^{i,t}(\{x(t)\})}{c^i(x(t))F_0^i([x(t),+\infty))+N^{i,t}([x(t),+\infty))} \cdot \frac{m^i(\{j\})+M^{i,j}(t)}{m^i(E)+\sum_{h\neq i} M^{i,h}(t)}.
\end{align*}
\end{theorem}
\begin{proof}
By Remark \ref{rem:equivSM}, for all $j\in E$ it is $$\Prob{S_{t+1}=j|S_{0:t}} = \Prob{L_{N(t+1)}=j|N(t), L_{0:N(t)}, T_{0:N(t)-1}, T_{N(t)}> l(t)}.$$
Hence, on the event $\{N(t)=n,L_{0:n}=i_{0:n},T_{0:n-1}=t_{0:n-1}\}$ with $j=i=i_n$ it is
\begin{align*}
\Prob{S_{t+1}=j|S_{0:t}} &= \Prob{T_n>x(t)| L_{0:n}=i_{0:n}, T_{0:n-1}=t_{0:n-1}, T_{n}> l(t)} \\
&= \frac{ \E{F^{i_n}((x(t),+\infty)) | L_{0:n}=i_{0:n}, T_{0:n-1}=t_{0:n-1}} }{ \E{F^{i_n}((l(t),+\infty)) | L_{0:n}=i_{0:n}, T_{0:n-1}=t_{0:n-1}}} \\
&=\frac{F^{i_n}_*((x(t),+\infty))}{F^{i_n}_*([x(t),+\infty))} = k_t(S_{0:t},j),
\end{align*}
where the third and fourth equalities follows from Theorem \ref{thm:conjugacy}. Similarly, on $\{N(t)=n,L_{0:n}=i_{0:n},T_{0:n-1}=t_{0:n-1}\}$ with $j\neq i= i_n$, $\Prob{S_{t+1}=j|S_{0:t}}$ equals 
\begin{align*}
&\Prob{T_n=x(t),L_{n+1}=j|L_{0:n}=i_{0:n}, T_{0:n-1}=t_{0:n-1}, T_{n}> l(t)} =\\
&\quad\quad=\frac{F_*^{i_{n}}(\{x(t)\})}{F_*^{i_{n}}([x(t),+\infty))}\cdot \Prob{L_{n+1}=j|L_{0:n}=i_{0:n}, T_{0:n-1}=t_{0:n-1}, T_{n}> l(t)} \\
&\quad\quad=\frac{F_*^{i_{n}}(\{x(t)\})}{F_*^{i_{n}}([x(t),+\infty))}\cdot \E{P^{i_n,j}  |L_{0:n}=i_{0:n}, T_{0:n-1}=t_{0:n-1}, T_{n}> l(t)}\\
&\quad\quad=\frac{F_*^{i_{n}}(\{x(t)\})}{F_*^{i_{n}}([x(t),+\infty))}\cdot \frac{m_*^i(\{j\})}{m_*^i(E)} =k_t(S_{0:t},j).
\end{align*}
\end{proof}

By the Ionescu-Tulcea Theorem \citep[Theorem 4.7]{Cinlar2011}, the sequence of predictive distributions $k_t$ defines the law of a new stochastic process: 
\begin{definition}
The process $S=(S_t)_{t\geq 0}$ is a \emph{reinforced semi-Markov process} with parameters $(\mathbf{m},\mathbf{c},\mathbf{F_0})$, or $S\sim RSM(\mathbf{m},\mathbf{c},\mathbf{F_0})$, if $\Prob{S_0=l_0}=1$ and $\Prob{S_{t+1}=j|S_{0:t}}=k_t(S_{0:t};j)$ for all $j\in E$ and $t\geq 0$.
\end{definition}

With this definition, the following is a trivial corollary of Theorem \ref{thm:pred}: 
\begin{corollary}\label{corollary:preddist}
If, conditionally on $(\mathbf{P},\mathbf{F})\sim SMBS(\mathbf{m},\mathbf{c},\mathbf{F_0})$, $S \sim SM(\mathbf{P},\mathbf{F})$, then marginally it is $S\sim RSM(\mathbf{m},\mathbf{c},\mathbf{F_0})$.
\end{corollary}

The process $S\sim RSM(\mathbf{m},\mathbf{c},\mathbf{F_0})$ is ``reinforced'' because if it performs a transition, this becomes more likely in the future \citep{Coppersmith1986,Pemantle1988,Pemantle2007}. This is because if $S_t=i$, then $k_{t}(S_{0:t},j)$ is increasing in $M^{i,j}(t)$, the number of times that a transition from $i$ to $j$ has already occurred by time $t$. 

\section{Predictive characterization by reinforced urn processes} \label{sec:urns}

We characterize semi-Markov beta-Stacy processes by means of reinforced urns. We build on the urn-based
characterizations of the 
Dirichlet process of \citet{Blackwell1973} and the beta-Stacy process of \citet{Muliere2000}.

\citet{Blackwell1973} characterized the Dirichlet process $ Dir(m)$ as the mixing measure of the sequence of colors extracted from a generalized P\`olya urn $U$. This is defined as follows: $U$ initially contains $m(\{i\})$ balls of color $i\in E$. Balls are repeatedly sampled from $U$. Each extracted ball is replaced together with an additional one of the same color. Denote with $(L_n)_{n\geq 1}$ the sequence of colors extracted from $U$. \citet{Blackwell1973} showed that, conditional on a random $P\sim Dir(m)$, the $L_n$ are independent and have common distribution $P$. 

\citet{Muliere2000} characterized the beta-Stacy process $BS(c,F_0)$ using the following urn process. Let $V=(V_k)_{k\geq 1}$ be a sequence of P\`olya urns. Each $V_k$ contains $c(t)F_0(\{t\})$ black balls and $c(t)F_0((t,+\infty))$ white balls. Every time a ball is extracted from an urn, it is replaced together with another of the same color. The urns $V_k$ generate a sequence $(T_n)_{n\geq 1}$ of random variables. Specifically, starting from $n=1$, the process proceeds as follows: beginning from $k=1$, a ball is sampled from $V_k$. If it's white, sampling continues from $V_{k+1}$, otherwise $T_n=k$. Once $T_n$ is generated, $T_{n+1}$ is determined by restarting from $V_1$. \citet{Muliere2000} showed that, conditional on some $F(\cdot)\sim BS(c,F_0)$,  the $T_n$ are independent and have distribution $F$.

\begin{definition}
We say that the generalized P\`olya urn $U$ that characterizes the $Dir(m)$ process is a \emph{$Dir(m)$-urn}. Similarly, we say that the system $V$ of reinforced urns $V_1$, $V_2$, $V_3$, $\ldots$ that characterizes the $BS(c,F_0)$ process is a \emph{$BS(c,F_0)$-system}.
\end{definition}

We now characterize the semi-Markov beta-Stacy process. We associate each $i\in E$ with a $Dir(m^i)$-urn $U_i$ and a $BS(c^i,F_0^i)$-system $V_i=(V_{i,k})_{k\geq 1}$. Generate $\{(L_k,T_k)\}_{k\geq 0}$ as follows: set $L_0=l_0$. Then, for all $k\geq 0$, generate $T_k$ from $V_{L_k}$, and, independently, set $L_{k+1}$ to the color extracted from $U_{L_k}$. This process is illustrated in Figure \ref{fig:urnproc}.

\begin{figure}
\centering
\includegraphics[scale=0.8,trim={0cm 6cm 10cm 0cm},clip]{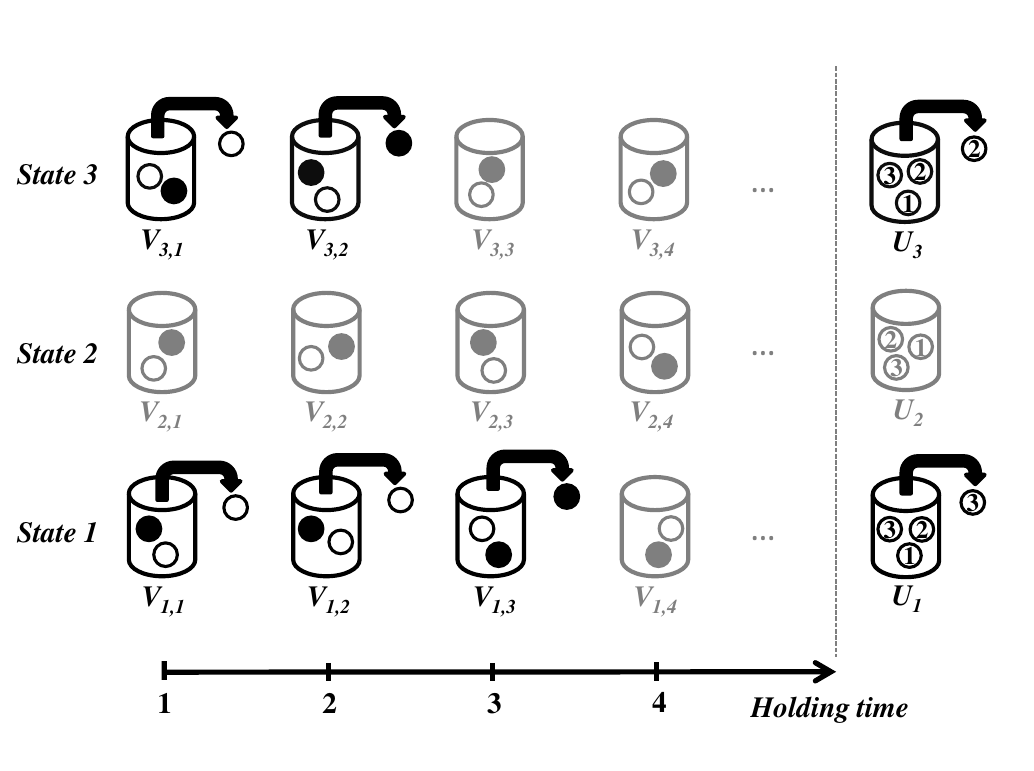}
\caption{Graphical illustration of the reinforced urn process of Section \ref{sec:urns}. In the figure, the path of the process corresponds to the observation of $(L_0,T_0)=(1,3)$, $(L_1,T_1)=(3,2)$, and $L_2=2$. Specifically, the process starts from the urn corresponding to $T_0=1$ for the holding time of the state $L_0=1$. The $BS(c^1,F_0^1)$-system $V_{11}$, $V_{12}$, $V_{13}$, $\ldots$ is traversed left to right until a black ball is extracted from $V_{13}$, determining $T_0=3$. The process then jumps to the $Dir(m^1)$-urn $U_1$, from which a ball of color ``3'' is extracted. Thus, $L_1=3$ and the process jumps to $V_{31}$, the first urn of the $BS(c^3,F_0^3)$-system represented in the third row of the graph. The process then resumes similarly to generate the values $T_1=2$ and $L_2=2$.}
\label{fig:urnproc}
\end{figure}

Continuing, define $S=(S_t)_{t\geq 0}$ as follows. Let $\tau_0=0$, $\tau_{n+1}=\sum_{h=0}^n T_h$ for all $n\geq 0$, and $N(t)=\sum_{n=1}^{+\infty}\Ind{\tau_n \leq t}$ for all integers $t\geq 0$. Lastly, define $S_{t}=L_{N(t)}$ for all $t\geq 0$. It holds that $\mathbb{P}(S_{t+1}=\cdot|S_{0:t})=k_{t}(S_{0:t},\cdot)$, where $k_t$ is the kernel in Theorem \ref{thm:pred}. Hence, $S\sim RSM(\mathbf{m},\mathbf{c},\mathbf{F_0})$. Any reinforced semi-Markov process can be generated in this way.

The process $S=(S_t)_{t\geq 0}$ is \emph{recurrent} if $(L_k)_{k\geq 0}$ visits every state in $E$ an infinite number of times with probability 1. If $S$ is recurrent, the time $v_{i,n}$ of the $n$-th visit of $(L_k)_{k\geq 0}$ to the state $i$ is a finite and non-negative for every $n\geq 1$ and $i\in E$. Let $T_{i,n}=T_{v_{i,n}}$ be the length of time $S$ stays in $i$ during its $n$-th visit, and $L_{i,n}=L_{v_{i,n}+1}$ the next state visited  by $S$ after its $n$-th visit to $i$. 

The following result is a partial converse of Corollary \ref{corollary:preddist}:
\begin{theorem}\label{thm:reprthm} 
Suppose $S\sim RSM(\mathbf{m},\mathbf{c},\mathbf{F_0})$ is recurrent. Then there exists a random characteristic couple $(\mathbf{P},\textbf{F})$ such that: 
\begin{enumerate}
\item conditional on $(\mathbf{P},\textbf{F})$, $S\sim SM(\mathbf{P},\textbf{F})$;
\item $(\mathbf{P},\textbf{F})\sim SMBS(\mathbf{m},\mathbf{c},\mathbf{F_0})$.
\end{enumerate}
\end{theorem}

To show this result we will use of the following lemma:

\begin{lemma}\label{lemma:predlemma}
Suppose $S\sim RSM(\mathbf{m},\mathbf{c},\mathbf{F_0})$ is recurrent. Then: 
\begin{enumerate}
\item the sequences $\{(L_{i,n},T_{i,n})\}_{n\geq 1}$ for $i\in E$ are independent;
\item the sequences $(L_{i,n})_{n\geq 1}$ and $(T_{i,n})_{n\geq 1}$ are independent for all $i\in E$;
\item conditional on a $P^i\sim Dir(m^i)$, the $L_{i,n}$ are independent and have law $P^i(\cdot)$;
\item conditional on a $F^i(\cdot)\sim BS(c^i,F_0^i)$, the $T_{i,n}$ are independent and have law $F^i(\cdot)$;
\item all the $P^i(\cdot)$ and $F^i(\cdot)$ are independent. 
\end{enumerate}
\end{lemma}

\begin{proof}[Proof of Lemma \ref{lemma:predlemma}]
To show points (1)-(4), note that: i) for all $i\in E$, the sequence $(L_{i,n})_{n\geq 1}$ is generated by $Dir(m^i)$-urn $U_i$; ii) for all $i\in E$, $(T_{i,n})_{n\geq 1}$ is generated by the $BS(c^i,F_0^i)$-system $V_i$;  iii) the outcomes of the urns $U_i$, $V_{1,i}$, $V_{i,2}$, $\ldots$, for all $i\in E$ are independent. To prove (5), since $(L_{i,n})_{n\geq 1}$ and $(T_{i,n})_{n\geq 1}$ are exchangeable, by the de Finetti representation theorem, $P^i(\cdot)=\mathbb{P}(L_{i,1}\in\cdot|\mcal{L}_i)$ and $F^i(\cdot)=\mathbb{P}(T_{i,1}\in\cdot|\mcal{T}_i)$, where $\mcal{L}_i$ and $\mcal{T}_i$ are, respectively, the tail $\sigma$-fields of $(L_{i,n})_{n\geq 1}$ and $(T_{i,n})_{n\geq 1}$ \citep[Chapter 1]{Kallenberg2006}. The thesis now follows since all $\mcal{L}_i$ and $\mcal{T}_i$, $i\in E$, are independent. 
\end{proof}

\begin{proof}[Proof of Theorem \ref{thm:reprthm}]
Take $P^i(\cdot)$ and $F^i(\cdot)$ for $i\in E$ as given by Lemma \ref{lemma:predlemma}. Define $\mathbf{P}=(P^i(\{j\}))_{i,j\in E}$ (note that $P^i(\{i\})=0$ almost surely since $m^i(\{i\})=0$) and $\mathbf{F}=\{F^i(\cdot):i\in E\}$. Conditional on $(\mathbf{P},\mathbf{F})$, $\{(L_k,T_k)\}_{k\geq 0}$ is a Markov renewal process with characteristic couple $(\mathbf{P},\mathbf{F})$. This is because $\mathbb{P}(L_0=l_0|(\mathbf{P},\mathbf{F}))=1$ by definition. Moreover, on the event $\{L_n=i, v_{i,k}=n\}$, $k\leq n$, it is
\begin{align*}
\mathbb{P}(L_{n+1}=j,T_{n}\leq t|L_{0:n},T_{0:n-1},(\mathbf{P},\mathbf{F})) &= \mathbb{P}(L_{i,k}=j,T_{i,k}\leq t|(\mathbf{P},\mathbf{F})) \\
&= P^i(\{j\})F^i(t).
\end{align*}
This concludes the proof.
\end{proof}

\section{Simulation study}\label{sec:applications}

To illustrate the semi-Markov beta-Stacy process in action, we conducted a simulation study based Example \ref{ex:factory}. 

\subsection{Description of the simulation study}\label{sec:dgd}

Following \citet[Sections 4.3]{Barbu2009}, we generated a single realization $s_{0:1,000}$ from the semi-Markov process $(S_t)_{t\geq 0}$ describing the day-by-day status of the factory from day 0 to day 1,000. The law of this process was determined by assuming that: i) $S_0=1$ (so the factory begins fully functional); ii) the transition matrix is 
\begin{equation}\label{eqn:transmatr}
\mathbf{P}=\left[
\begin{array}{ccc}
0 & 1 & 0 \\ 
0.95 & 0 & 0.05 \\ 
1 & 0 & 0
\end{array} 
\right];
\end{equation}
 iii) $F^{1}(\cdot)$ is the geometric distribution $F^1_{\textrm{true}}(\{t\})=p(1-p)^{t-1}$, $t\geq 1$, with parameter $p=0.8$; iv) $F^2(\cdot)$ is the first-type discrete Weibull distribution $F^2_\textrm{true}(t)=1-q^{t^{k}}$, $t\geq 1$, of \citet{Nakagawa1975} with parameters $q=0.3$ and $k=0.5$; v) $F^3(\cdot)$ is the first-type discrete Weibull distribution $F^3_\textrm{true}(\cdot)$ with parameters $q=0.6$ and $k=0.9$. The observed sequence $s_{0:1,000}$ was considered as data to perform posterior inferences. 

\subsection{Prior specification}\label{sec:priornp}

We assign a semi-Markov beta-Stacy prior distribution $SMBS(\mathbf{m},\mathbf{c},\mathbf{F_0})$ to the data-generating characteristic couple $(\mathbf{P},\mathbf{F})$. We consider the measures $m^1(\cdot)$, $m^2(\cdot)$, and $m^3(\cdot)$ on $E=\{1,2,3\}$ determined by the conditions $m^i(\{1,2,3\})$ $=m^1(\{2\})$ $=m^2(\{1\})$ $=m^2(\{3\})$ $=m^3(\{1\})=1$ for all $i\in E$ (so $P^{2,1}$ and $P^{2,3}$ are marginally uniform over $(0,1)$). For all $i=1,2,3$, $F_0^i(\cdot)$ is the geometric distribution with parameter $p=0.3$. For all $i\in E$, we set $c^i(t)=c$ for all $t\geq 1$. We fix $c=0.1$, $1$, or $10$ separately. 

\subsection{Posterior distributions}

Figure \ref{fig:post} shows the plots of the posterior mean of $F^2(\cdot)$, together with a sample of 500 samples from the corresponding distribution. Posterior distributions were obtained from Theorem \ref{thm:conjugacy} using data $s_{0:M}$ with $M=0$ (so the posterior coincides with the prior), $M=100$, or $M=1,000$ (so whole simulated path is used). For comparison, the figure also reports the data-generating distribution of the holding-times of the state 2, i.e. of the time elapsed until either the tank is repaired or the factory has to stop after a failure. 

Figure \ref{fig:post} highlights how the posterior distribution obtained from the semi-Markov beta-Stacy prior is able to recover the underlying data-generating distribution by flexibly adapting to the observations, even when these deviate from prior assumptions. This is true both for data reflecting a short ($M=100$) or long ($M=1,000$) period of observation. The figure also highlights the impact of the concentration parameters $c$. As this increases, the dispersion of the distribution of $F^2(\cdot)$ around its mean decreases.

\begin{figure}
\centering
\includegraphics[scale=0.6]{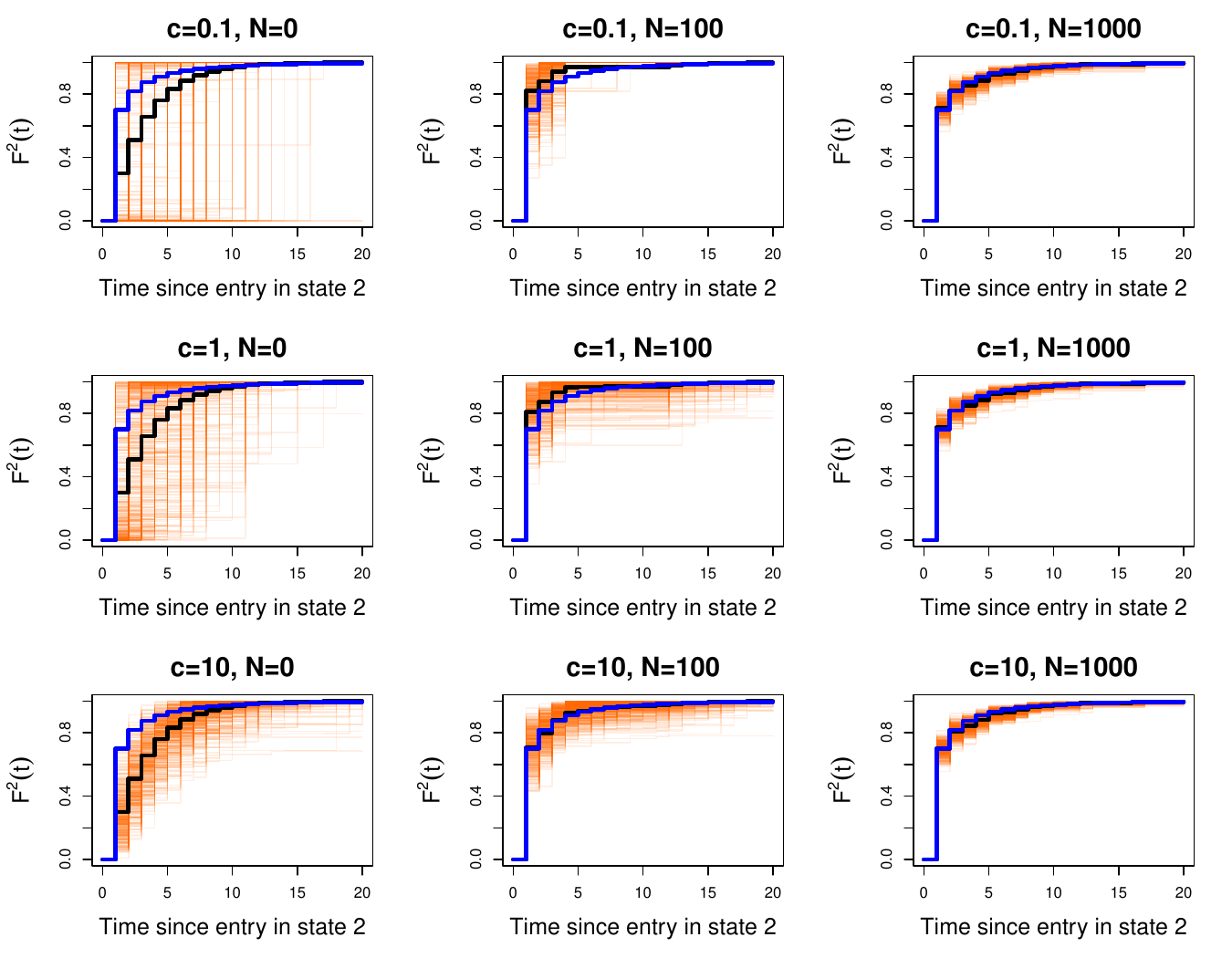}
\caption{Plot of the posterior distribution of $F^{2}(\cdot)$ for the semi-Markov process priors of Section \ref{sec:applications}. Results are shown for different values of: i) the prior concentration parameter $c$, which specifies the weight assigned to the prior centering distributions $F_0^2(\cdot)$; ii) the length $M$ of the observation period during which data $S_{0:M}$ is collected (if $M=0$, the posterior distribution coincides with the prior). Blue lines: true data-generating distribution $F^2(\cdot)$ (see Section \ref{sec:dgd}). Black lines: posterior mean of $F^2(\cdot)$. Orange lines: graph of 500 samples from the posterior distribution of $F^2(\cdot)$.}
\label{fig:post}
\end{figure}

\subsection{Predictive distributions}

Figure \ref{fig:pred} reports the estimates of the predictive distributions $P_h(j)=\mathbb{P}(S_{1,000+h}=j|S_{0:1,000}=s_{0:1,1000})$ obtained from the semi-Markov beta-Stacy prior with $c=1$ for all $h=1,\ldots,100$ and all $j=1,2,3$. These were obtained by simulating  $10^5$ future paths $(S_{1,000+h})_{h=1,\ldots,100}$ given the past observation of $S_{0:1,000}=s_{0:1,000}$ by sampling from the reinforced semi-Markov kernels of Corollary \ref{corollary:preddist}. Then, $P_h(j)$ was estimated as the proportion of simulations in which $S_{1,000+h}=j$.

Figure \ref{fig:pred} shows how the the $P_h(j)$ adapt over time as $h$ increases for all $j=1,2,3$, whose values stabilize in the long run. Specifically, for large $h$ the vector $(P_h(1)$, $P_h(2)$, $P_h(3))$ remain close to the limiting distribution $(\nu_1,\nu_2,\nu_3)$ of the data-generating semi-Markov process. This is obtained from Proposition 3.9 of \citet{Barbu2009} as $\nu_j=e_j m_j/ \sum_{i=1}^3 e_i m_i$, where $(e_1,e_2,e_3)=(\frac{1}{2.05},\frac{1}{2.05},\frac{0.05}{2.05})$ is the equilibrium distribution of $\mathbf{P}$, while $m_j=\sum_{t=0}^{+\infty}(1-F^j(t))$ is the expected sojourn time in the state $j$. 

\begin{figure}
\centering
\includegraphics[scale=0.6]{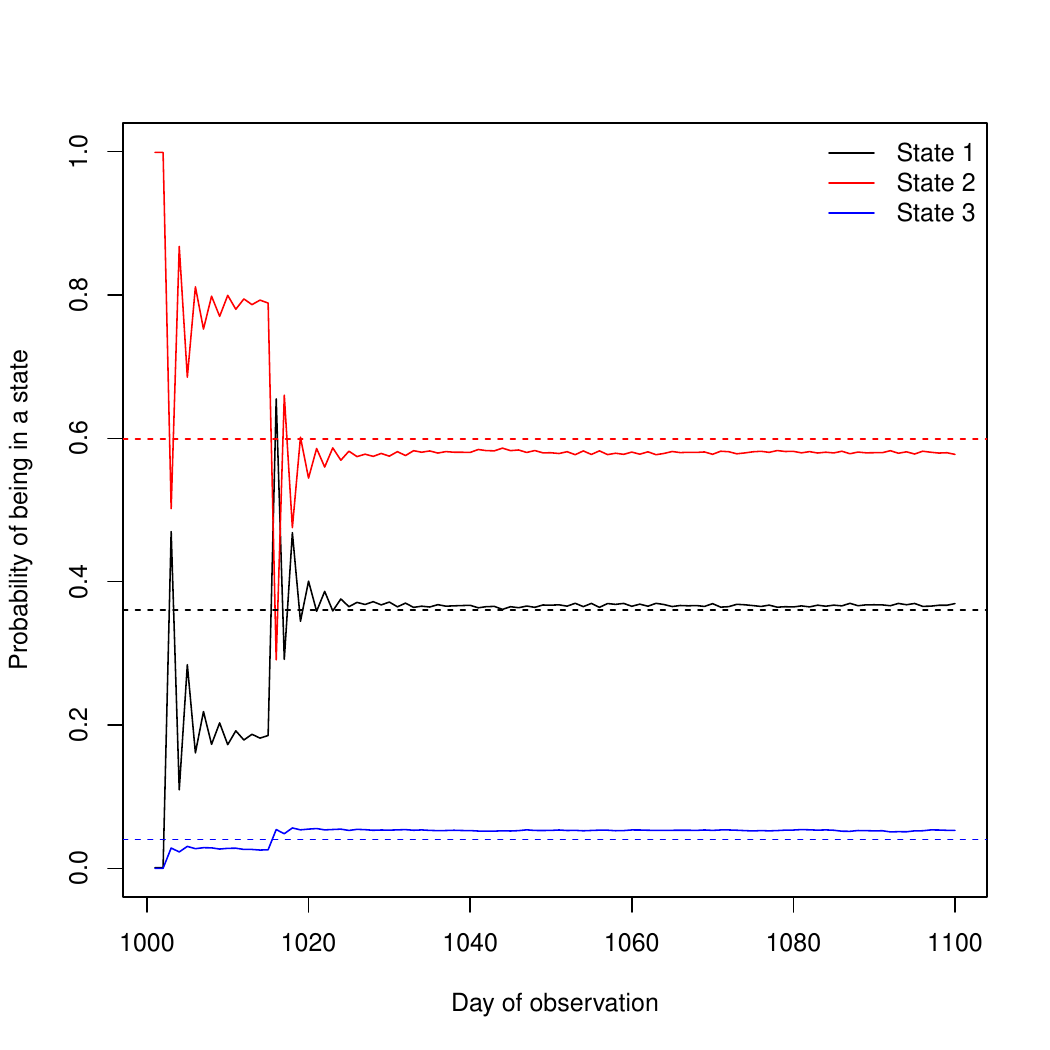}
\caption{Plot of the predictive probabilities $P_h(j)=\mathbb{P}(S_{1,000+h}=j|S_{0:1,000}=s_{0:1,1000})$ obtained from the semi-Markov beta-Stacy process of Section \ref{sec:applications} with $c=1$ for all $h=1,\ldots,100$. The value $P_h(j)$ is the probability that the factory will be in state $j=1,2,3$ after $h$ days in the future given its past history $S_{0:1000}$. The black, red, and blue lines are, respectively, $P_h(1)$, $P_h(2)$, and $P_h(3)$. The dashed lines represent the limiting distribution of the underlying data-generating semi-Markov process.} 
\label{fig:pred}
\end{figure}

\subsection{Robustness comparison}\label{sec:rob}

In additional simulations, we compare our non-parametric model of Section \ref{sec:priornp} with a parametric specification. The alternative model was defined as follows: for the transition matrix $\mathbf{P}$, we specify the same prior distribution as in Section \ref{sec:priornp}; for each $j=1,2,3$, we assume that $F^j(\cdot)$ is the geometric distribution with parameter $p_j$; each $p_j$ has an independent prior distribution uniform on $(0,1)$. Both the parametric model for $F^2(\cdot)$ and the non-parametric centering law $F^2_0(\cdot)$ have a different functional form than $F^2_\textrm{true}(\cdot)$. 

We iterated the following steps 10,000 times: i) we simulated a realization $s_{0:M}$ of the semi-Markov process of Section \ref{sec:dgd}; we considered $M=10, 100, 1,000,$ and $10,000$ separately; ii) for both models, we computed the posterior mean $\widehat{F}^2(\cdot)$ of $F^2(\cdot)$ given $s_{0:M}$, i.e. the Bayesian estimate of $F^2_\textrm{true}(\cdot)$ under the squared-error loss; iii) we computed the Kolmogorov-Smirnov distance $\max_{t=0,\ldots,20}\left|\widehat{F}^2(t) - F^2_\textrm{true}(t)\right|$ between $\widehat{F}^2(\cdot)$ and $F^2_\textrm{true}(\cdot)$. 

Figure \ref{fig:param} reports the distribution of the Kolmogorv-Smirnov metrics obtained in the simulations. Using our non-parametric approach, the distribution of the Kolmogorov-Smirnov distances concentrate around zero as the number of observations $M$ increases. This suggests that the semi-Markov beta-Stacy posterior mean can approach the true $F^2_\textrm{true}(\cdot)$ even though its centering $F^2_0(\cdot)$ is misspecified. On the other hand, the posterior mean of the misspecified parametric model does not seem to approach $F^2_\textrm{true}(\cdot)$, as the Kolmogorov-Smirnov metrics tend to concentrate away from zero even for large $M$.  

\begin{figure}
\centering
\includegraphics[scale=1]{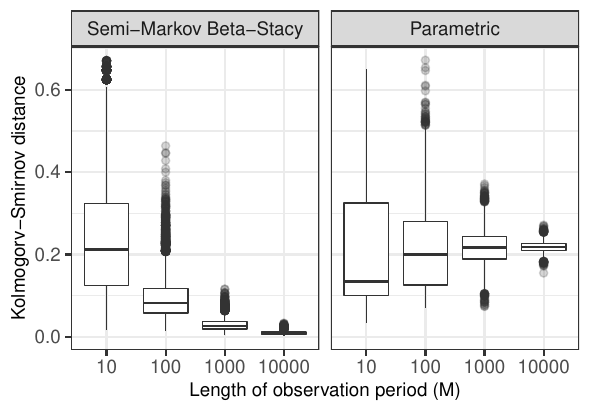}
\caption{Summary box plots of the Kolmogorv-Smirnov metrics obtained in the simulations of Section \ref{sec:rob}. Left panel: non-parametric semi-Markov beta-Stacy model. Right panel: parametric model based on the geometric distribution.}
\label{fig:param}
\end{figure}

\section{Concluding remarks} \label{sec:discussion}

We introduced the semi-Markov beta-Stacy process, a Bayesian nonparametric process prior for semi-Markov models. We characterized it from a predictive perspective by ``piecing together'' different reinforced urn models,  each characterizing simpler processes. This approach provides a fresh strategy for the specification of Bayesian nonparametric models for complex processes \citep{Muliere2000, Muliere2003}. 

The semi-Markov beta-Stacy is amenable to several generalizations. In Definition \ref{def:renewal}, the holding time $T_k$ depends only on the current state $L_k$ and not on the following state $L_{k+1}$. More generally, $T_k$ may depend on both $L_k$ and $L_{k+1}$ \citep{Barbu2009}. This can be represented by substituting the distribution $F^i$ in Definition \ref{def:renewal} with one of the form $F^{i,j}$ and letting $\mathbf{F}=(F^{i,j}(\cdot):i,j\in E, i\neq j)$. \citet{Arfe2020} generalizes our definitions and results to this case.

In addition, in the urn-based construction of Section \ref{sec:urns}, each extracted ball may be reinforced by a fixed or random amount of balls of the same or different colors \citet{Muliere2006}. This could allow a finer control of the level of uncertainty attached to the urns' initial composition, i.e. to the prior centering distribution \citep{Arfe2018a}. 

To induce dependence across components of the prior,  urns other than the sampled one could be reinforced as well. This interaction among urns could lead to interesting models, in which observations provide indirect information about distributions that have not generated them \citep{Paganoni2004,Muliere2005}.

We are studying a regression model in which the distribution of the holding times and the transition matrices depend on a vector of covariates. As in \citet{Arfe2018a}, this is done by letting the initial composition of the urns be a function of the covariates and some additional parameters. Such model could be used for the analysis of multi-stage diseases in medical studies \citep{Barbu2004, Mitchell2011}. 

Finally, we are applying the semi-Markov beta-Stacy process to perform inference and predictions in Hidden Semi-Markov Models (HSMMs). In these models, the sequence of visited states is observed only indirectly \citep[Chapter 6]{Barbu2009}. As a specific application, we are developing a novel approach for changepoint analysis in which the state of a semi-Markov process represents the latent regimen of a time series \citep{Smith1975,Muliere1985,Ko2015,Peluso2018}. 

\section*{Acknowledgments}
The Authors have no conflicts of interest to declare. AA would like to thank Sarah Craver for her helpful suggestions.

\bibliography{bibliography}
\bibliographystyle{rss}

\end{document}